\documentclass[11pt,reqno]{amsart}

\usepackage[hyperfootnotes=false,pdfpage layout=SinglePage]{hyperref}
\usepackage{amssymb}
\usepackage{amsmath}
\usepackage{amsthm}
\usepackage{amsfonts}
\usepackage{latexsym}
\usepackage{mathrsfs}
\usepackage{titlesec}
\usepackage{algorithm}
\usepackage{algorithmic}
\usepackage{caption}[2007/09/01] 

\newtheorem{thm}{Theorem}[section]

\newtheorem{lem}[thm]{Lemma}

\theoremstyle{definition}

\numberwithin{equation}{section}
\titleformat{\section}{\normalfont\bfseries\centering}{\thesection.}{.25em}{}
\titleformat{\subsection}{\normalfont\bfseries}{\thesubsection.}{.25em}{}
\allowdisplaybreaks

\newcommand{\R}{\ensuremath{\mathbb R}}    
\newcommand{\C}{\ensuremath{\mathbb C}}    
\newcommand{\N}{\ensuremath{\mathbb N}}    
\newcommand{\Z}{\ensuremath{\mathbb Z}}    
\newcommand{\T}{\mathbb{T}}                

         \newcommand{\frakP}{\mathfrak P}

\newcommand{\vphi}{\varphi}

\newcommand{\vek}[2]
{
   \begin{pmatrix}
      #1\\
      #2
   \end{pmatrix}
}

\renewcommand{\Im}{\operatorname{Im}}
\renewcommand{\Re}{\operatorname{Re}}

\newcommand{\Lra}{\Longrightarrow}

\newcommand{\ol}{\overline}

\newcommand{\rmref}[1]{{\rm\ref{#1}}}

\begin{document}
\title[Phase Retrieval from 4{\it N}--4 Measurements]{Phase Retrieval from 4{\it N}--4 Measurements}  

\author{Friedrich Philipp}
\address{Institut f\"ur Mathematik, Technische Universit\"at Berlin, Stra\ss e des 17.\ Juni 136, 10623 Berlin, Germany}
\email{philipp@math.tu-berlin.de}
\urladdr{www.tu-berlin.de/?fmphilipp}

\begin{abstract}
We prove by means of elementary methods that phase retrieval of complex polynomials $p$ of degree less than $N$ is possible with $4N-4$ phaseless Fourier measurements of $p$ and $p'$. In addition, we provide an associated algorithm and prove that it recovers $p$ up to global phase.
\end{abstract}



\maketitle
\thispagestyle{empty}

\section{Introduction}
Phase retrieval is the recovery (up to a global phase factor) of signals from intensity measurements, i.e.\ from the absolute values of (scalar) linear measurements of the signal. This particularly challenging task is motivated by applications in, e.g., X-ray crystallography \cite{fi1,m}, optical design \cite{fa}, and quantum mechanics \cite{c,j,lr}. In many applications it is usually the Fourier transform of the signal from which the phase gets lost. While in practice one tries to overcome the resulting under-determination by exploiting a priori information on the phase, mathematicians and engineers are also interested in the question whether phase retrieval (without prior knowledge) is possible when the number of measurements is increased. In fact, there are several kinds of phase retrieval problems. Among them are the following two:

\begin{itemize}
\item[(a)] Which systems of vectors allow for phase retrieval of ``generic'' signals and how large are these?
\item[(b)] Which systems of vectors allow for phase retrieval of {\em every} signal and how large are these?
\end{itemize}

Here, we restrict ourselves to the finite-dimensional complex case\footnote{For phase retrieval in $\R^N$, question (b) has been answered completely in \cite{bce}.}, i.e., phase retrieval in $\C^N$, where $N\ge 2$. While for (a) at least the second question has been answered (the answer is $2N$, cf.\ \cite{fi2}), the questions in (b) have remained unanswered to date. However, there are results which give us an idea of possible answers. Let us briefly recall the history of problem (b). In 2004, it has been shown in \cite{fi2} that systems as in (b) must contain at least $3N-2$ vectors. In 2011 (Arxiv version), this lower bound could be improved to $4N - 2\alpha - 3$, where $\alpha$ denotes the number of ones in the binary representation of $N-1$, see \cite{hmw}. In \cite{bce} the authors proved that a ``generic'' system consisting of $4N-2$ vectors allows for phase retrieval of every signal in $\C^N$. Moreover, it was shown in \cite{bcmn} (see also \cite{bce}) that a system as in (b) must necessarily have the so-called complement property, meaning that if some subsystem is not spanning $\C^N$, then its complement is.

Recently, it has been conjectured in \cite{bcmn} (the ($4N$--\,$4$)-Conjecture) that phase retrieval is never possible with less than $4N-4$ measurements, and that ``generic'' measurement systems consisting of $4N-4$ vectors allow for phase retrieval. This would correspond to the real case where this statement holds true with the lower bound $2N-1$ (cf.\ \cite{bce}). And indeed, the second part of the ($4N$--\,$4$)-Conjecture has very recently been proved by means of algebraic geometry in \cite{cehv}. The first part is true for the dimensions $N=2$ and $N=3$, cf.\ \cite{bcmn}.

However, in order to understand the full picture, one seeks for {\it structured} measurement systems with $4N-4$ vectors allowing for phase retrieval. The first example of this kind was given by Bodmann and Hammen in \cite{bh}. They constructed a system of $4N-4$ vectors in the $N$-dimensional linear space $\frakP_N$ of complex polynomials of degree less than $N$ and proved that it allows for phase retrieval, cf.\ \cite[Theorem 2.3]{bh}. Since the second part of the ($4N$--\,$4$)-Conjecture was not proved when Bodmann and Hammen published their work, their result was the first showing that phase retrieval with $4N-4$ measurements is possible. Another measurement ensemble consisting of $4N-4$ vectors and allowing for phase retrieval was provided in \cite{fmnw}. This system consists of the magnitudes of DCT-like measurements of the signal and a modulated version of it. In their proof, the authors make heavily use of the so-called circular autocorrelation. A recovery algorithm is presented as well.

In the present contribution, we will prove a variant of the main theorem in \cite{bh}. The measurements in \cite{bh} are in fact intensities of polynomial evaluations at points on the unit circle $\T$ and on another circle intersecting $\T$. Our main theorem is the following.

\begin{thm}\label{t:main}
Let $p$ be a polynomial with complex coefficients of degree at most $N-1$, and let
$$
w_1,\ldots,w_{2N-1}\in\T
\qquad\text{as well as}\qquad
z_1,\ldots,z_{2N-3}\in\T
$$
be mutually distinct points on the unit circle, respectively. Then the $4N-4$ intensity measurements
\begin{equation}\label{e:measurements}
|p(w_j)|,\;j=1,\ldots,2N-1,
\qquad\text{and}\qquad
|p'(z_k)|,\;k=1,\ldots,2N-3,
\end{equation}
determine $p$ uniquely, up to a global phase factor.
\end{thm}

Note that point evaluations of $p$ and $p'$ are linear measurements of $p$. Hence, the measurements in \eqref{e:measurements} are indeed intensity measurements of $p$.

Theorem \ref{t:main} bears two advantages over Theorem 2.3 in \cite{bh}. First, its proof (given in Section \ref{s:proof}) is self-contained and simpler than that of \cite[Theorem 2.3]{bh}. Second, the linear measurements in Theorem \ref{t:main} are evaluations of polynomials at points {\em on the unit circle only} and hence correspond to Fourier measurements\footnote{Hereby, we mean the scalar product in $\C^N$ with a vector $\big(1,\omega^j,\omega^{2j},\ldots,\omega^{(N-1)j}\big)^T$, where $|\omega| = 1$.} in $\C^N$. However, we remark that the second set of measurements in Theorem \ref{t:main} consists of Fourier measurements of $p'$ and not of $p$ itself.

The rest of this paper consists of two parts. In Section \ref{s:proof} we prove Theorem \ref{t:main}. In the second part (Section \ref{s:alg}) we present a corresponding algorithm in Section \ref{s:alg} (Algorithm 1) and prove in Theorem \ref{t:alg} that it in fact recovers any polynomial in $\frakP_N$ from $4N-4$ intensity measurements as in Theorem \ref{t:main}, up to global phase.

\vspace{.5cm}
\section{Proof of Theorem \ref{t:main}}\label{s:proof}
Recall that a trigonometric polynomial of degree at most $m$ has the form
\begin{equation}\label{e:tp}
f(t) = \sum_{n=-m}^m\alpha_n e^{int} = \sum_{n=1}^m\alpha_{-n}e^{-int} + \alpha_0 + \sum_{n=1}^{m}\alpha_ne^{int},\quad t\in\R,
\end{equation}
where $\alpha_n\in\C$, $n=-m,\ldots,m$. The trigonometric polynomial $f$ is said to be {\em real} if $\alpha_{-n} = \ol{\alpha_n}$ for $n=0,\ldots,m$. For the proof of Theorem \ref{t:main} we only need the following two simple lemmas. Although Lemma \ref{l:interpolation1} directly follows from \cite[Satz 10.6]{sw}, we prove it here for the convenience of the reader.

\begin{lem}\label{l:interpolation1}
A trigonometric polynomial $f$ of degree at most $m$ is uniquely determined by any $2m+1$ mutually distinct point evaluations of $f$ in $[0,2\pi)$.
\end{lem}
\begin{proof}
Let $f$ be as in \eqref{e:tp} and put
$$
p(z) := \sum_{n=-m}^m\alpha_n z^{n+m} = \sum_{k=0}^{2m}\alpha_{k-m}z^k,\quad z\in\C.
$$
Now, let $t_1,\ldots,t_{2m+1}\in [0,2\pi)$ be $2m+1$ distinct real points so that $f(t_k)$ is known for $k=1,\ldots,2m+1$. Then, as $p(e^{it_k}) = e^{imt_k}f(t_k)$, $k=1,\ldots,2m+1$, the values of the polynomial $p$ are known at $2m+1$ distinct points. These determine $p$ uniquely since the degree of $p$ is at most $2m$. Hence, from $f(t) = e^{-imt}p(e^{it})$ we obtain $f$.
\end{proof}

\begin{lem}\label{l:real}
Let $f$ and $g$ be two real trigonometric polynomials such that $|f(t)| = |g(t)|$ for all $t\in\R$. Then $g = -f$ or $g = f$.
\end{lem}
\begin{proof}
It is no restriction to assume $f\neq 0$ and $g\neq 0$. Since the values $f(t)$ and $g(t)$ are real for each $t\in\R$ and $f$ and $g$ are continuous, it follows that there exists $t_0 > 0$ such that on $\Delta = (0,t_0)$ we have either $g|_\Delta = -f|_\Delta$ or $g|_\Delta = f|_\Delta$. But as $f$ and $g$ are restrictions of entire functions to $\R$ it follows that $g = -f$ or $g = f$.
\end{proof}

The above two simple lemmas now allow us to prove Theorem \ref{t:main}.

\begin{proof}[Proof of Theorem \rmref{t:main}]
First of all, we observe that $t\mapsto |p(e^{it})|^2$ and $t\mapsto |p'(e^{it})|^2$ are trigonometric polynomials of degrees at most $N-1$ and $N-2$, respectively. Thus, by Lemma \ref{l:interpolation1} the measurements in Theorem \ref{t:main} uniquely determine the restrictions of $|p|$ and $|p'|$ to the unit circle. Hence, we have to prove that two polynomials $p\in\frakP_N$ and $q\in\frakP_N$ with
\begin{equation}\label{e:voraussetzung}
|p(e^{it})| = |q(e^{it})|
\quad\text{and}\quad
|p'(e^{it})| = |q'(e^{it})|
\quad\text{for all }t\in\R
\end{equation}
must be linearly dependent. Without loss of generality we assume that both polynomials $p$ and $q$ do not vanish identically. Now, we define functions $g_p,g_q : \C\to\C$ by
$$
g_p(z) := zp'(z)\ol{p(z)}
\quad\text{and}\quad
g_q(z) := zq'(z)\ol{q(z)}, \quad z\in\C.
$$
Then \eqref{e:voraussetzung} implies
\begin{equation}\label{e:simplest_ever}
|g_p(e^{it})| = |g_q(e^{it})|
\end{equation}
for all $t\in\R$. But also
\begin{equation}\label{e:almost_simple}
\Im g_p(e^{it}) = \Im g_q(e^{it})
\end{equation}
for all $t\in\R$ since we have
$$
\frac{d}{dt}|p(e^{it})|^2
= 2\Re\left(ie^{it}p'(e^{it})\ol{p(e^{it})}\right) = -2\Im g_p(e^{it}),
$$
and, similarly, $\frac d{dt}|q(e^{it})|^2 = -2\Im g_q(e^{it})$. As a consequence of \eqref{e:simplest_ever} and \eqref{e:almost_simple} we obtain
\begin{equation}\label{e:close_to_simple}
|\Re g_p(e^{it})| = |\Re g_q(e^{it})|
\end{equation}
for all $t\in\R$. But $\Re g_p(e^{it})$ and $\Re g_q(e^{it})$ are {\em real} trigonometric polynomials in $t$, so that Lemma \ref{l:real} implies that $\Re g_q(e^{it}) = \Re g_p(e^{it})$ for all $t\in\R$ or that $\Re g_q(e^{it}) = -\Re g_p(e^{it})$ for all $t\in\R$. Combining this and \eqref{e:almost_simple}, it follows that either 
\begin{equation}\label{e:alt}
g_q(e^{it}) = g_p(e^{it})\quad\text{for $t\in\R$}\qquad\text{or}\qquad g_q(e^{it}) = -\ol{g_p(e^{it})}\quad\text{for $t\in\R$}.
\end{equation}
In the first case, we have
\begin{equation}\label{e:toshow}
p'(z)\ol{p(z)} = q'(z)\ol{q(z)}\quad\text{for all }z\in\T.
\end{equation}
Taking \eqref{e:voraussetzung} into account, we see that multiplication of \eqref{e:toshow} with $p(z)q(z)$, $z\in\T$, leads to
$$
p'(z)q(z) = q'(z)p(z)\quad\text{for all }z\in\T.
$$
This implies $p'(z)q(z) = q'(z)p(z)$ for all $z\in\C$ and thus $\frac{d}{dz}\frac{p(z)}{q(z)}\equiv 0$. Consequently, $p$ and $q$ are linearly dependent.

Let us assume that the second case in \eqref{e:alt} applies, i.e.
\begin{equation}\label{e:const}
e^{it}q'(e^{it})\ol{q(e^{it})} = -e^{-it}\ol{p'(e^{it})}p(e^{it})\quad\text{for all }t\in\R.
\end{equation}
We will show that both $p$ and $q$ must be constant and therefore linearly dependent. For this, let
$$
p(z) = \sum_{k=0}^{N-1}\alpha_kz^k
\quad\text{and}\quad
q(z) = \sum_{k=0}^{N-1}\beta_kz^k.
$$
Then the relation \eqref{e:const} reads
$$
\sum_{k=0}^{N-1}\sum_{j=1}^{N-1} j\beta_j\ol{\beta_k}e^{i(j-k)t} = -\sum_{k=0}^{N-1}\sum_{j=1}^{N-1} j\ol{\alpha_j}\alpha_ke^{i(k-j)t}\quad\text{for all }t\in\R.
$$
If we now compare the zero-th coefficients (i.e.\ those for $j=k$) on right and left hand side of the previous equation, we obtain
$$
\sum_{j=1}^{N-1}j|\beta_j|^2 = -\sum_{j=1}^{N-1}j|\alpha_j|^2.
$$
Thus, $\beta_1 = \ldots = \beta_{N-1} = \alpha_1 = \ldots = \alpha_{N-1} = 0$, which implies $p(z) = \alpha_0$ and $q(z) = \beta_0$.
\end{proof}

\vspace{.5cm}
\section{A Reconstruction Algorithm}\label{s:alg}
We say that two polynomials $p$ and $q$ are equivalent (in terms of global phase), and write $p\sim q$, if $q = e^{i\vphi}p$ with some $\vphi\in [0,2\pi)$. Theorem \ref{t:main} shows that for any two ensembles
$$
w_1,\ldots,w_{2N-1}\in\T
\qquad\text{and}\qquad
z_1,\ldots,z_{2N-3}\in\T
$$
of mutually distinct points on the unit circle, respectively, the (well-defined) {\em magnitude map}
$$
\frakP_N/\sim\,\to\R^{2N-1}\times\R^{2N-3},\quad [p]\mapsto \left(\left(\left|p(w_j)\right|\right)_{j=1}^{2N-1},\left(\left|p'(z_j)\right|\right)_{j=1}^{2N-3}\right),
$$
is injective. However, Theorem \ref{t:main} does not provide an inverse map, nor is its proof constructive.

Algorithm 1 below reconstructs a signal $p\in\frakP_N$ (up to a global phase factor) from the following $4N-4$ phaseless measurements (where $\omega_m := e^{\frac{2\pi i}{m}}$ for $m\in\N$):
\begin{equation}\label{e:roots}
\left|p(\omega_{2N-1}^j)\right|,\; j=0,\ldots,2N-2
\qquad\text{and}\qquad
\left|p(\omega_{2N-3}^j)\right|,\; j=0,\ldots,2N-4.
\end{equation}
Before we present the algorithm, let us prove a couple of preparatory lemmas. The following one is a well known interpolation result (see also the proof of Theorem 2.3 in \cite{bh}). For the convenience of the reader we provide its short proof below.

\begin{lem}\label{l:interpolation2}
Let $f$ be a trigonometric polynomial of degree at most $N-1$, i.e.
$$
f(t) = \sum_{n=-N+1}^{N-1}a_n e^{int},\quad t\in\R,
$$
where $a_n\in\C$, $n=-N+1,\ldots,N-1$. Moreover, put $t_j = \frac{2\pi j}{2N-1}$, $j=0,\ldots,2N-2$. Then we have
$$
a_n = \frac{1}{2N-1}\sum_{j=0}^{2N-2}f(t_j)e^{-int_j},\quad n=-N+1,\ldots,N-1.
$$
In particular, the values $f(t_j)$, $j=0,\ldots,2N-2$, uniquely determine $f$.
\end{lem}
\begin{proof}
This fact is easily checked by using the identity $\sum_{j=0}^{m-1}e^{\frac{2\pi ilj}{m}} = 0$, which holds for integers $l,m\in\Z$, $l/m\notin\Z$, and follows directly from the well known formula
$$
\sum_{j=0}^{m-1}z^j = \frac{1 - z^m}{1-z},\quad z\neq 1,
$$
with $z = e^{\frac{2\pi il}{m}}$.
\end{proof}

\begin{lem}
Let $p\in\frakP_N$, $p(z) = \sum_{k=0}^{N-1}\alpha_kz^k$, and put
\begin{equation}\label{e:fns1}
f_n := \frac{1}{2N-1}\sum_{j=0}^{2N-2}|p(\omega_{2N-1}^j)|^2\exp\left(-\frac{2\pi ijn}{2N-1}\right),\quad n=0,\ldots,N-1,
\end{equation}
as well as
\begin{equation}\label{e:fn's1}
f_n' := \frac{1}{2N-3}\sum_{j=0}^{2N-4}|p'(\omega_{2N-3}^j)|^2\exp\left(-\frac{2\pi ijn}{2N-3}\right),\quad n=0,\ldots,N-2.
\end{equation}
Moreover, set $d := \deg(p)$. Then we have\footnote{Here and in the following, sums $\sum_{j=m}^n$ with $m > n$ should be read as zero.}
\begin{align}
\label{e:p3}|p(e^{it})|^2 &= \sum_{n=1}^{d}f_n e^{int} + f_0 + \sum_{n=1}^{d}\ol{f_n}e^{-int},\quad t\in\R,\\
\label{e:p'3}|p'(e^{it})|^2 &= \sum_{n=1}^{d-1}f_n' e^{int} + f_0' + \sum_{n=1}^{d-1}\ol{f_n'}e^{-int},\quad t\in\R.
\end{align}
as well as
\begin{align}
\label{e:fnd}f_n = \sum_{\ell = 0}^{d-n}\ol{\alpha_\ell}\alpha_{\ell+n},\quad n=0,\ldots,d,\qquad &\text{and}\quad f_n = 0\text{ if }n > d\\
\label{e:fn'd}f_n' = \sum_{\ell=1}^{d-n}\ell(\ell+n)\ol{\alpha_\ell}\alpha_{\ell+n},\quad n=0,\ldots,d-1,\qquad &\text{and}\quad f_n' = 0\text{ if }n > d-1.
\end{align}
\end{lem}
\begin{proof}
The formulas \eqref{e:p3} and \eqref{e:p'3} (with $d$ replaced by $N-1$) follow directly from Lemma \ref{l:interpolation2}. On the other hand, we also have
$$
|p(e^{it})|^2 = \sum_{k,n=0}^{N-1}\ol{\alpha_k}\alpha_ne^{i(n-k)t}
\qquad\text{and}\qquad
|p'(e^{it})|^2 = \sum_{k,n=1}^{N-1}kn\ol{\alpha_k}\alpha_ne^{i(n-k)t}.
$$
This implies that
\begin{align}
\label{e:fn1}f_n &= \sum_{\ell=0}^{N-1-n}\ol{\alpha_\ell}\alpha_{\ell+n},\quad n=0,\ldots,N-1,\\
\label{e:fn'1}f_n' &= \sum_{\ell=1}^{N-1-n}\ell(\ell+n)\ol{\alpha_\ell}\alpha_{\ell+n},\quad n=0,\ldots,N-2.
\end{align}
Now, if $\alpha_j = 0$ if $j > d$, we obtain \eqref{e:fnd}--\eqref{e:fn'd}, and thus also \eqref{e:p3}--\eqref{e:p'3}.
\end{proof}

In the following, it is convenient to define $f_{N-1}' := 0$.

\begin{lem}\label{l:breakthrough}
Let $p\in\frakP_N$, $p(z) = \sum_{k=0}^{N-1}\alpha_kz^k$, assume that $p$ is not a constant, and let $f_n$ and $f_n'$ be defined as in \eqref{e:fns1}--\eqref{e:fn's1}. Let
$$
k := \max\{n : f_n\neq 0\}\quad\text{and}\quad k' := \max\{n : f_n'\neq 0\},
$$
and set
\begin{equation}\label{e:m}
m := -\frac{k}{2} + \sqrt{\frac{f_k'}{f_k} + \frac{k^2}{4}}.
\end{equation}
Then $k\,\ge\,k'$, $m\in\N_0$, and we have
\begin{align}
\begin{split}\label{e:alphas}
\alpha_j &= 0\quad\text{for }j < m\\
\alpha_m&\neq 0,\quad \alpha_{m+k}\neq 0\\
\alpha_j &= 0\quad\text{for }j > m+k.
\end{split}
\end{align}
In particular, $\deg(p) = m+k$.
\end{lem}
\begin{proof}
Let $d := \deg(p)\le N-1$. Then formulas \eqref{e:fnd}--\eqref{e:fn'd} imply that $k,k'\le d$. First of all, we prove that
\begin{equation}\label{e:implication}
k\,\le\,d-1\quad\Lra\quad\alpha_0 = \ldots = \alpha_{d-1-k} = 0,\;\alpha_{d-k}\neq 0.
\end{equation}
Towards an induction argument, assume that $\alpha_j = 0$ if $j < n$, where $0\le n\le d-k-1$ (which is true for $n=0$). Then $d - n > k$, and thus \eqref{e:fnd} yields $0 = f_{d-n} = \ol{\alpha_n}\alpha_d$, which implies $\alpha_n = 0$ as $\alpha_d\neq 0$. Hence, $\alpha_0 = \ldots = \alpha_{d-1-k} = 0$ is proved. $\alpha_{d-k}\neq 0$ now follows from \eqref{e:fnd}: $0\neq f_k = \ol{\alpha_{d-k}}\alpha_d$.

To prove $k\ge k'$, we observe that if $k+1 > d-1$, then $f_n' = 0$ for $n\ge k+1$ follows from \eqref{e:fn'd}. Let $k+1\le d-1$. Then \eqref{e:implication} implies $\alpha_j = 0$ for $j < d-k$. Hence, for $n\ge k+1$, we have $d-k > d-n$ and thus $f_n' = 0$ (see \eqref{e:fn'd}). Therefore, in both cases, $f_n' = 0$ for $n\ge k+1$, which implies $k\ge k'$.

Let $k > k'$. Then $f_k' = 0$ and thus also $m = 0$. Suppose that $d\ge k+1$. Then \eqref{e:implication} implies that $\alpha_0 = \ldots = \alpha_{d-1-k} = 0$. In addition, from $f_k' = 0$ we further obtain
$$
0 = \sum_{\ell=1}^{d-k}\ell(\ell+k)\ol{\alpha_\ell}\alpha_{\ell+k} = (d-k)d\ol{\alpha_{d-k}}\alpha_d,
$$
and thus $\alpha_{d-k} = 0$. We conclude
$$
f_k = \sum_{\ell = 0}^{d-k}\ol{\alpha_\ell}\alpha_{\ell+k} = 0,
$$
contrary to the definition of $k$. Therefore, $k = d$ follows. This also yields $f_k = \ol{\alpha_0}\alpha_k$, which implies $\alpha_0\neq 0$ and $\alpha_k\neq 0$. This proves \eqref{e:alphas} in the case $k > k'$.

Let $k = k'$. We define $r := d - k$. As $f_n' = 0$ if $n\ge d$, it follows that $k = k'\le d-1$. Thus, $r\ge 1$, and by \eqref{e:implication} we have
$$
\alpha_0 = \ldots = \alpha_{r-1} = 0\quad\text{and}\quad\alpha_r\neq 0.
$$
Thus, if we can show that $r = m$ (where $m$ is as defined in \eqref{e:m}), then also \eqref{e:alphas} follows, and the lemma is proved. To prove $r = m$, we obtain from \eqref{e:fnd} and \eqref{e:fn'd} that
$$
f_k = \ol{\alpha_{d-k}}\alpha_d\quad\text{and}\quad f_k' = (d-k)d\ol{\alpha_{d-k}}\alpha_d.
$$
Therefore,
$$
\frac{f_k'}{f_k} = (d-k)d = r(r+k) = r^2 + kr = \left(r + \frac k 2\right)^2 - \frac{k^2}{4},
$$
which implies $r = m$.
\end{proof}

\newpage
\makeatletter
  \fst@algorithm\@fs@pre
\makeatother
\begingroup
\captionsetup{style=ruled,type=algorithm,skip=0pt}
\caption{Phase Retrieval from the Measurements in \eqref{e:roots}}
\endgroup
\makeatletter
  \@fs@mid\vspace{2pt}
\makeatother
\begin{algorithmic}
\STATE {\bf Input:} The data \eqref{e:roots} obtained from $p\in\frakP_N$, $p(z) = \sum_{k=0}^{N-1}\alpha_kz^k$.
\STATE {\bf Output:} A polynomial $q$ such that $q = e^{it}p$ for some $t\in\R$.\bigskip
\STATE With the given data \eqref{e:roots} define $f_n$ for $n=0,\ldots,N-1$ and $f_n'$ for $n=0,\ldots,N-2$ as in \eqref{e:fns1}--\eqref{e:fn's1}.
\IF{$f_n' = 0$ for all $n = 0,\ldots,N-2$}
\RETURN $q = \sqrt{f_0}$
\ENDIF
\begin{equation}\label{e:km_alg}
k := \max\{n : f_n\neq 0\}
,\quad m := -\frac{k}{2} + \sqrt{\frac{f_k'}{f_k} + \frac{k^2}{4}},
\quad\text{and}\quad \delta_{m,m+k} := f_{k}.
\end{equation}
\STATE {\bf if} $k\ge 1$ {\bf then}
\begin{equation}\label{e:erstes_alg}
\vek{\delta_{m,m+k-1}}{\delta_{m+1,m+k}} :=
\frac{1}{2m+k}
\begin{pmatrix}
(m+1)(m+k) & -1\\
-m(m+k-1) & 1
\end{pmatrix}
\vek{f_{k-1}}{f_{k-1}'}.
\end{equation}
\STATE {\bf end if}
\IF{$k\ge 2$}{
\FOR{$n=2$ \TO $k$}
\FOR{$\ell=m+1$ \TO $m+n-1$}
\STATE Set
$$
\delta_{\ell,\ell+k-n} := \frac{\delta_{\ell,m+k}}{\delta_{m,m+k}}\,\delta_{m,\ell+k-n}.
$$
\ENDFOR
\STATE Put
\begin{align}
\alpha &:= f_{k-n} - \sum_{\ell=m+1}^{m+n-1}\delta_{\ell,\ell+k-n}\label{e:Alpha}\\
\beta &:= f_{k-n}' - \sum_{\ell=m+1}^{m+n-1}\ell(\ell+k-n)\delta_{\ell,\ell+k-n}\label{e:Beta}
\end{align}
and
\begin{equation}\label{e:last}
\vek{\delta_{m,m+k-n}}{\delta_{m+n,m+k}} :=
\frac{1}{n(2m+k)}
\begin{pmatrix}
(m+n)(m+k) & -1\\
-m(m+k-n) & 1
\end{pmatrix}
\vek\alpha\beta.
\end{equation}
\ENDFOR}
\ENDIF
\STATE With $r := \sqrt{\delta_{m,m}}$ ($\delta_{m,m}$ will be a positive nunber) define the polynonial
\begin{equation}\label{e:ende}
q(z) := \sum_{j=m}^{m+k}\frac{\delta_{m,j}}{r}z^j.
\end{equation}
\RETURN $q$
\end{algorithmic}
\makeatletter
  \vspace{2pt}\@fs@post
\makeatother

\newpage
Let us now prove that Algorithm 1 indeed recovers $p$ up to global phase.

\begin{thm}\label{t:alg}
Let $p\in\frakP_N$, $p(z) = \sum_{n=0}^{N-1}\alpha_nz^n$. Then Algorithm {\rm 1} recovers $p$, up to a global phase factor.
\end{thm}
\begin{proof}
The initial {\bf if} query checks on whether $p$ is constant and returns this constant if this is the case. In the following we assume that $p$ is not a constant. It then follows from Lemma \ref{l:breakthrough} that with $k$ and $m$ in \eqref{e:km_alg} we have $d := \deg(p) = m+k$. Next, we shall show that Algorithm 1 defines $\delta_{j,k}$ such that
\begin{equation}\label{e:deltaalpha}
\delta_{i,j} = \ol{\alpha_i}\alpha_j, \quad i,j=m,\ldots,m+k,\;j\ge i.
\end{equation}
First of all, we have (cf.\ \eqref{e:alphas} and \eqref{e:fnd})
$$
\delta_{m,m+k} = f_k = \sum_{\ell=m}^{d-k}\ol{\alpha_\ell}\alpha_{\ell+k} = \ol{\alpha_m}\alpha_{m+k}\neq 0.
$$
If $k=0$, then this proves \eqref{e:deltaalpha}. Let $k\ge 1$. Then
$$
\begin{pmatrix}
1 & 1\\
m(m+k-1) & (m+1)(m+k)
\end{pmatrix}
\vek{\ol{\alpha_m}\alpha_{m+k-1}}{\ol{\alpha_{m+1}}\alpha_{m+k}}
= \vek{f_{k-1}}{f_{k-1}'}.
$$
And as the $2\times 2$-matrix in \eqref{e:erstes_alg} is the inverse of the one above, we have proved $\delta_{i,j} = \ol{\alpha_i}\alpha_j$ for $j-i\ge k-1$. If $k=1$, this implies \eqref{e:deltaalpha}. Assume that $k\ge 2$. In the following, we shall proceed by induction. We assume that $\delta_{i,j} = \ol{\alpha_i}\alpha_j$ holds for $j-i\ge k-s$, where $s\ge 1$. Let us prove that $\delta_{i,j} = \ol{\alpha_i}\alpha_j$ is then also true for $j-i = k-s-1$. Define $n := s+1\ge 2$. Then, due to the algorithm, for $\ell=m+1,\ldots,m+n-1$ we have
$$
\delta_{\ell,\ell+k-n} = \frac{\delta_{\ell,m+k}}{\delta_{m,m+k}}\,\delta_{m,\ell+k-n} = \frac{\ol{\alpha_\ell}\alpha_{m+k}}{\ol{\alpha_m}\alpha_{m+k}}\cdot\ol{\alpha_m}\alpha_{\ell+k-n} = \ol{\alpha_\ell}\alpha_{\ell+k-n}.
$$
For $\alpha$ and $\beta$, defined in \eqref{e:Alpha}--\eqref{e:Beta}, this implies
$$
\alpha = f_{k-n} - \sum_{\ell=m+1}^{m+n-1}\ol{\alpha_\ell}\alpha_{\ell+k-n} = \ol{\alpha_m}\alpha_{m+k-n} + \ol{\alpha_{m+n}}\alpha_{m+k}.
$$
Similarly, one gets
$$
\beta = m(m+k-n)\ol{\alpha_m}\alpha_{m+k-n} + (m+n)(m+k)\ol{\alpha_{m+n}}\alpha_{m+k}.
$$
Therefore, it follows from \eqref{e:last} that $\ol{\alpha_m}\alpha_{m+k-n} = \delta_{m,m+k-n}$ and $\ol{\alpha_{m+n}}\alpha_{m+k} = \delta_{m+n,m+k}$, and \eqref{e:deltaalpha} is proved.

Since $p$ is only unique up to global phase, we may assume without loss of generality that $\alpha_m\in (0,\infty)$. Hence, if $r$ is defined as in the algorithm ($r = \sqrt{\delta_{m,m}}$) we see from \eqref{e:deltaalpha} that $r = \alpha_m$ and thus $\alpha_\ell = \delta_{m,\ell}/r$ for $\ell=m+1,\ldots,m+k$. Hence, the polynomial $q$ defined in  \eqref{e:ende} is a unimodular multiple of $p$.
\end{proof}


\vspace{.5cm}
\begin{thebibliography}{99}
%
\bibitem{bce}  R.\ Balan, P.G.\ Casazza, and D.\ Edidin,
               On signal reconstruction without phase,
               Appl.\ Comp.\ Harmon.\ Anal.\ {\bf 20} (2006) 345--356.

\bibitem{bcmn} A.\ Bandeira, J.\ Cahill, D.\ Mixon, and A.\ Nelson,
               Saving phase: Injectivity and stability for phase retrieval,
               to appear in Appl.\ Comp.\ Harmon.\ Anal.

\bibitem{bh}   B.G.\ Bodmann and N.\ Hammen,
               Stable phase retrieval with Low-Redundancy Frames,
							 Preprint, arXiv:1302.5487v1.

\bibitem{cehv}9 A.\ Conca, D.\ Edidin, M.\ Hering, C.\ Vinzant,
               An algebraic characterization of injectivity in phase retrieval,
							 Preprint, arXiv:1312.0158.

\bibitem{c}    J.V.\ Corbett,
               The Pauli problem, state reconstruction and quantum-real numbers,
               Rep.\ Math.\ Phys.\ {\bf 57} (2006) 53--68.

\bibitem{fa}   M.W.\ Farn,
               New iterative algorithm for the design of phase-only gratings,
               Proc.\ SPIE 1555, SPIE, Bellingham, WA, 1991, pp.\ 3442.

\bibitem{fmnw} M.\ Fickus, D.G.\ Mixon, A.A.\ Nelson, and Y.\ Wang,
               Phase retrieval from very few measurements,
							 to appear in Linear Algebra Appl., arXiv:1307.7176.

\bibitem{fi1}  J.R.\ Fienup,
               Phase retrieval algorithms: a comparison,
               Applied Optics {\bf 21} (1982) 2758--2769.

\bibitem{fi2}  J.\ Finkelstein,
               Pure-state informationally complete and ``really'' complete measurements,
               Phys.\ Rev.\ A {\bf 70} (2004) 052107.

\bibitem{hmw}  T.\ Heinosaari, L.\ Mazzarella, M.M.\ Wolf,
               Quantum tomography under prior information,
							 Commun.\ Math.\ Phys.\ {\bf 318} (2013) 355--374.

\bibitem{j}    P.\ Jaming, Uniqueness results for the phase retrieval problem of fractional Fourier transforms of variable order,
               arXiv:1009.3418.

\bibitem{lr}   A.I.\ Lvovsky and M.G.\ Raymer,
               Continuous-variable optical quantum state tomography,
               Reviews of Modern Physics {\bf 81} (2009) 299--332.

\bibitem{m}    R.P.\ Millane,
               Phase retrieval in crystallography and optics,
               J.\ Opt.\ Soc.\ Am.\ A.\ {\bf 7} (1990) 394--411.

\bibitem{sw}   R.\ Schaback, H.\ Wendland,
               Numerische Mathematik,
							 5th edition, Springer, Berlin, Heidelberg, 2005.
\end{thebibliography}
\end{document}